\documentclass[11pt,letterpaper,oneside]{article}
\pdfoutput=1
\usepackage[left=1.3in,right=1.3in,bottom=1.25in,top=1.25in]{geometry}
\usepackage[utf8]{inputenc}
\usepackage{amsmath,amsfonts,amssymb,amsthm}%
\usepackage{setspace}
\usepackage[misc,geometry]{ifsym} 
\usepackage{microtype}
\usepackage{gensymb}
\usepackage{inconsolata}
\usepackage[normalem]{ulem}

\newtheorem{theorem}{Theorem}[section]%
\newtheorem{proposition}[theorem]{Proposition}%
\newtheorem{lemma}[theorem]{Lemma}%
\newtheorem{corollary}[theorem]{Corollary}%
\newtheorem{question}{Question}%
\newtheorem{definition}[theorem]{Definition}%

\DeclareMathOperator{\Succ}{Succ}

\DeclareMathOperator{\CL}{CL}

\DeclareMathOperator{\ACL}{ACL}
\DeclareMathOperator{\DCL}{DCL}
\DeclareMathOperator{\cl}{cl}

\DeclareMathOperator{\dcl}{dcl}
\DeclareMathOperator{\acl}{acl}
\DeclareMathOperator{\FINITE}{Fin}
\DeclareMathOperator{\Struct}{CompStr}

\newcommand{\cF}{\EM{\mathcal{F}}}

\newcommand{\Lang}{\EM{\mathcal{L}}}
\newcommand{\KLang}{\EM{\mathcal{K}}}
\newcommand{\ty}{\EM{\mathrm{ty}}}

\newcommand{\pars}{\,\cdot\,}

\renewcommand{\hat}{\widehat}

\def\st{\,:\,}

\def\M{{\EM{\mathcal{M}}}}

\def\cM{{\EM{\mathcal{M}}}}

\def\cN{{\EM{\mathcal{N}}}}
\def\cG{{\EM{\mathcal{G}}}}
\def\cY{{\EM{\mathcal{Y}}}}
\def\cZ{{\EM{\mathcal{Z}}}}

\def\cA{{\EM{\mathcal{A}}}}

\def\cP{{\EM{\mathcal{P}}}}
\def\cQ{{\EM{\mathcal{Q}}}}
\def\cR{{\EM{\mathcal{R}}}}
\def\cK{{\EM{\mathcal{K}}}}

\def\x{{\EM{\ol{x}}}}
\def\xx{{\EM{\ol{x}}}}
\newcommand\dd{{\EM{\ol{d}}}} 
\def\y{{\EM{\ol{y}}}}
\def\yy{{\EM{\ol{y}}}}

\def\zz{{\EM{\ol{z}}}}

\def\a{{\EM{\ol{a}}}}

\def\aa{{\EM{\ol{a}}}}

\def\b{{\EM{\ol{b}}}}
\def\bb{{\EM{\ol{b}}}}
\def\bbzero{{\EM{\ol{b_0}}}}
\def\bbone{{\EM{\ol{b_1}}}}
\def\yyzero{{\EM{\ol{y_0}}}}
\def\yyone{{\EM{\ol{y_1}}}}

\def\Lomega#1{{\EM{L_{#1, \w}}}}
\def\Lww{\Lomega{\w}}

\def\const#1{\EM{U_{#1}}}
\def\dconst#1{\EM{V_{#1}}}

\newcommand{\defn}[1]{{\bf{#1}}}
\newcommand{\defas}{{\EM{\ :=\ }}}

\def\w{\EM{\omega}}

\def\Nats{{\EM{{\mbb{N}}}}}

\def\Integers{{\EM{{\mbb{Z}}}}}

\newcommand{\True}{\mathrm{True}}
\newcommand{\False}{\mathrm{False}}
\newcommand{\TrueFalse}{\{\True, \False\}}

\def\^{\EM{{}^{\And}}}

\def\Or{\EM{\vee}}
\def\And{\EM{\wedge}}

\def\<{\EM{\langle}}
\def\>{\EM{\rangle}}

\def\indent{\hspace*{2em}}

\def\nl{\newline}

\def\EM#1{\ensuremath{#1}}
\def\mbb#1{\EM{\mathbb{#1}}}
\def\mbf#1{\EM{\mathop{\pmb{#1}}}}

\def\ol#1{\EM{\overline{#1}}}

\def\ul#1{\underline{#1}}

\newcommand{\Prime}[1]{p_{#1}}

\newcommand{\zerojump}[1]{\mbf{0}^{(#1)}}
\newcommand{\zerosinglejump}{\mbf{0}'}

\usepackage{graphicx}

\makeatletter
\newcommand*{\bigcdot}{}
\DeclareRobustCommand*{\bigcdot}{%
  \mathbin{\mathpalette\bigcdot@{}}%
}
\newcommand*{\bigcdot@scalefactor}{.7}
\newcommand*{\bigcdot@widthfactor}{1.15}
\newcommand*{\bigcdot@}[2]{%
  \sbox0{$#1\vcenter{}$}
  \sbox2{$#1\cdot\m@th$}%
  \hbox to \bigcdot@widthfactor\wd2{%
    \hfil
    \raise\ht0\hbox{%
      \scalebox{\bigcdot@scalefactor}{%
        \lower\ht0\hbox{$#1\bullet\m@th$}%
      }%
    }%
    \hfil
  }%
}
\makeatother

\title{On computable aspects of algebraic and definable closure}

\author{
Nathanael Ackerman\\
\small Harvard University\\
\small Cambridge, MA 02138, USA\\
\small \texttt{nate@aleph0.net}
\and
Cameron Freer\\
\small Massachusetts Institute of Technology\\
\small Cambridge, MA 02139, USA\\
\small \texttt{freer@mit.edu}
\and
Rehana Patel\\
\small African Institute for Mathematical Sciences\\
\small M'bour--Thi\`es, Senegal\\
\small \texttt{rpatel@aims-senegal.org}
}
\date{}

\begin{document}
\maketitle

\begin{abstract}
We investigate the computability of algebraic closure and definable closure with respect to a collection of formulas. 
We show that for a computable collection of formulas of quantifier rank at most $n$, in any given computable structure, both algebraic and definable closure with respect to that collection are $\Sigma^0_{n+2}$ sets. We further show 
that these bounds are tight.
\end{abstract}

\hspace{10pt}
\emph{keywords:} algebraic closure, definable closure, computable model theory


\section{Introduction}
\label{Section:Introduction}

In this paper we study 
the computability-theoretic content of two model-theoretic concepts: algebraic closure and definable closure.
These notions are fundamental to model theory, and have been studied explicitly in various contexts
\cite{AFP, MR1683298, MR3580180}.

In order to study the computable model theory of these notions, we 
consider iteratively-defined algebraic and definable closure operators with respect to a specified set of formulas, and focus on certain associated sets.
For a set of formulas $\Phi$ and a structure $\cN$, we define sets $\acl_{\Phi, \cN}$ and $\dcl_{\Phi, \cN}$, which capture the information contained in the operators for algebraic and definable closure in $\cN$ with respect to formulas in $\Phi$. We also define sets $\ACL_n$ and $\DCL_n$, which describe the computable information that is already present in the first step of this iterative process,
for first-order formulas of quantifier rank at most $n$. 

The paper is organized as follows.
In Section \ref{Section:Preliminaries}, we provide definitions of $\ACL_n$, $\DCL_n$, $\acl_{\Phi, \cN}$, and $\dcl_{\Phi, \cN}$  and establish some basic relationships among them.
Section~\ref{Section:Upper Bounds} 
gives upper bounds on the computability-theoretic strength of these objects in the quantifier-free case --- namely, of $\ACL_0$, $\DCL_0$, 
$\acl_{\Phi, \cN}$, and $\dcl_{\Phi, \cN}$ where $\Phi$ is a computable set of quantifier-free formulas. 
Section~\ref{Section:Lower Bounds} gives corresponding lower bounds on these objects, which establish tightness of the upper bounds;
for
$\ACL_0$ and $\DCL_0$, tightness is achieved
via structures that are model-theoretically ``nice'', namely, are $\aleph_0$-categorical or of finite Morley rank.
Finally, in Section~\ref{Section:Boolean Combinations} we use these results to provide bounds on the computational strength of $\ACL_n$, $\DCL_n$, $\acl_{\Phi, \cN}$ and $\dcl_{\Phi, \cN}$ for arbitrary $n$ and computable collections $\Phi$ of formulas of quantifier rank $n$.

Parts of this paper appeared in an extended abstract
\cite{compACL-LFCS} presented at LFCS 2020.


\section{Preliminaries}
\label{Section:Preliminaries}
In this section we introduce some terminology and notation, and 
define the main objects of study of this paper: 
 $\ACL_n$, $\DCL_n$, $\acl_{\Phi, \cN}$, and $\dcl_{\Phi, \cN}$. 
 We prove some basic relationships among them, and discuss their connection with standard model-theoretic notions of algebraic and definable closure.

For standard notions from computability theory, see, e.g., \cite{MR3496974}.
We write $\{e\}(n)$ to represent the output of the $e$th Turing machine run on input $n$, if it converges, and in this case write $\{e\}(n)\!\downarrow$.
Define $W_e \defas \{ n \in \Nats \st \{e\}(n)\!\downarrow\}$ and 
$\FINITE \defas \{e \in \Nats \st W_e\text{ is finite}\}$. Recall that $\FINITE$ is $\Sigma^0_2$-complete \cite[Theorem~4.3.2]{MR3496974}.

In this paper we will focus on computable languages that are relational.
Note that this leads to no loss of generality due to the standard fact that computable languages with function or constant symbols can be interpreted computably in relational languages where there is a relation for the graph of each function.  For the definitions of languages, first-order formulas, and structures, see \cite{MR1221741}.

In the context of algebraic and definable closure, we will often consider formulas with a specified partition of their free variables, which we write with a semicolon, e.g., $\varphi(\x; \y)$.
When we refer to a set of formulas, we mean a set of formulas with specified variable partitions.

We will work with many-sorted languages and structures; for more details, see \cite[\S1.1]{MR2908005}.
Let $\Lang$ be a (many-sorted) language, let $\cN$ be an $\Lang$-structure, and suppose that $\aa$ is a tuple of elements of $\cN$.
We say that the \defn{type} of $\aa$ is 
$\prod_{i \leq n} X_{i}$
when $\aa \in \prod_{i \leq n} (X_i)^{\cN}$,
where each of $X_0, \ldots, X_{n-1}$ is a sort of $\Lang$.
The type of a tuple of variables is the product of the sorts of its constituent variables (in order). The type of a relation symbol is defined to be the type of the tuple of its free variables, and similarly for formulas.
We write $(\forall \xx: X)$ and $(\exists \xx: X)$ to quantify over a tuple of variables $\xx$ of type $X$ (which includes the special case of a single variable of a given sort).

If one wanted to avoid the use of many-sorted languages, one could instead encode
each sort using a unary relation symbol --- and indeed this would not affect most of our results. However, in Section~\ref{Section:Lower Bounds} we are interested in how model-theoretically complicated the structures we build are,
and 
the single-sorted version of the
construction in Proposition~\ref{Lower bound on ACL} 
would no longer yield
an $\aleph_0$-categorical structure.

A \emph{graph} is a pair $(V,E)$ where $V$ is a set of vertices and $E$ is a
symmetric irreflexive binary relation on $V$.
A \emph{chain} in a graph is a cycle-free connected component of the graph each of whose vertices has degree $1$ or $2$; hence a chain either is finite with at least two vertices, or is infinite on one side (an \emph{$\Nats$-chain}), or is infinite on both sides (a \emph{$\Integers$-chain}). The \emph{order} of a chain is the number of vertices in the chain.	

Similarly, a \emph{digraph} is a pair $(V, E)$ where $V$ is a set of vertices and $E$ is an asymmetric binary relation on $V$, i.e., no vertices have self-loops and any two vertices have an edge in at most one direction.
A \emph{path} in a digraph is a connected component of the graph, containing at least one edge, in which each vertex has in-degree at most $1$ and out-degree at most $1$, and having a (necessarily unique) vertex with in-degree $0$.
A vertex of a path is \emph{initial} if it has in-degree $0$ and \emph{final} if it has out-degree $0$. Hence a path either is finite with a unique initial and unique final vertex, or is infinite with a unique initial vertex (an \emph{$\Nats$-path}). The \emph{order} of a path is the number of vertices in the path.

We now define computable languages and structures.

\begin{definition}
Suppose 
$\Lang  = \bigl((X_j)_{j\in J},$ $(R_i)_{i \in I})$
is a language,
where $I,J \in \Nats \cup \{\Nats\}$
and $(X_j)_{j \in J}$ and $(R_i)_{i \in I}$ are collections of sorts and relation symbols, respectively.
Let $\ty_\Lang \colon  I \to J^{<\w}$
be such that
for all $i \in I$, we have
$\ty_\Lang (i) = (j_0, \ldots, j_{n-1})$ where the type of $R_i$ is
$\prod_{k < n} X_{j_k}$.
We say that $\Lang$ is a \defn{computable language} when $\ty_\Lang$
is a computable function.
For each computable language, we fix a computable encoding of all first-order formulas of the language.

A \defn{computable $\Lang$-structure} $\cN$ is an $\Lang$-structure with computable underlying set such that
the sets $\{(a, j) \st a \in  (X_j)^\cN\}$ and 
$\{(\bb, i) \st \bb \in (R_i)^\cN\}$ are computable subsets of the appropriate domains.

We say that $c \in \Nats$ is a \defn{code for a structure} if $\{c\}(0)$ is a code for a computable language (via some fixed enumeration of functions of the form $\ty_\Lang$)
and $\{c\}(1)$ is a code for some computable structure in that language.
In this case, we write $\Lang_c$ for the language that $\{c\}(0)$ codes,
$\cM_c$ for the structure that $\{c\}(1)$ codes, and 
$T_c$ for the first-order theory of $\cM_c$. Let $\Struct$ be the collection of $c\in\Nats$ that are codes for structures.
\end{definition}

Note that these notions relativize in the obvious way. For more details on basic notions in computable model theory, see 
\cite{fokina_harizanov_melnikov_2014, MR1673621}.

Towards defining sets that capture the computable content of algebraic and definable closure, we first describe when a formula is algebraic or definable at a given tuple.

\begin{definition}
Let $\Lang$ be a language,
let $\varphi(\xx;\yy)$ be a first-order $\Lang$-formula, and let $\cN$ be an $\Lang$-structure. Suppose $\aa\in \cN$ has the same type as $\xx$.
The formula $\varphi(\xx;\yy)$ is \defn{algebraic at} $\aa$ if 
\[
\bigl\{\bb\in \cN \st \cN \models \varphi(\aa;\bb)\bigr\}
\]
is finite (possibly empty), and
\defn{definable at} $\aa$ if 
this set 
is a singleton.
\end{definition}

We now describe several sets that 
capture the information contained in a single step of the process of determining algebraic or definable closure.

\begin{definition}
\begin{itemize}
\item[]
\item $\CL \defas \Bigl\{(c, \varphi(\xx; \yy), \aa, k)
\st c \in \Struct$,\   $\varphi(\xx;\yy)$ \emph{is a first-order $\Lang_c$-formula,\  $\aa \in \cM_c$ has the same type
as $\xx$,\  and $k \in \Nats \cup \{\infty\}$ is such that}
\hbox{$\bigl| \{\bb\in \cM_c \st \cM_c \models \varphi(\aa;\bb)\} \bigr| = k\Bigr\}$.}
\vspace*{5pt}

\item $\ACL \defas \bigl\{(c, \varphi(\xx; \yy), \aa) \st \emph{\text{there~exists}}\ k \in \Nats \emph{\text{~with}}\ (c, \varphi(\xx; \yy), \aa, k)$ $\in \CL\bigr\}$. 
\vspace*{5pt}

\item $\DCL \defas \bigl\{(c, \varphi(\xx; \yy), \aa) \st (c, \varphi(\xx; \yy), \aa, 1) \in \CL\bigr\}$. 
\vspace*{5pt}

\item For $Y \in \{\CL, \ACL, \DCL\}$ and $n \in \Nats$ let
\[
Y_n \defas \{t \in Y \st \mathrm{the~second~coordinate~of~} t \mathrm{~is~a~Boolean~combination~of~}
\Sigma_n\text{-}\mathrm{formulas}\}.
\]

\item For $Y \in \{\CL, \ACL, \DCL\} \cup \{\CL_n, \ACL_n, \DCL_n\}_{n \in \Nats}$ and $c\in \Struct$, let 
\[
Y^c \defas \{u \st (c)\^u \in Y\},
\]
i.e., select those elements of $Y$ whose first coordinate is $c$, and then remove this first coordinate.
\end{itemize}
\end{definition}

Note that $\Struct$ is a $\Pi^0_2$ set. Hence
the sets $\CL, \ACL, \DCL$ are already 
complicated from the 
computability-theoretic perspective.
As such,
when we consider the complexity of whether formulas are algebraic or definable at various tuples, 
we will consider the question of how complex $\CL^c, \ACL^c, \DCL^c$
can be, when $c$ is a code for a structure. The next three lemmas connect these sets. 

\begin{lemma}
\label{CL for algebraic arguments is c.e. in DCL}
Uniformly in the parameters $c\in\Struct$ and $n\in\Nats$, the set
\[
\bigl\{(\varphi(\xx;\yy), \aa, k) \in \CL_n^c \st k \in \Nats,\ k\ge 1\bigr\}
\]
is $\Sigma^0_1$ in $\DCL_n^c$.
\end{lemma}
\begin{proof}
Suppose $\varphi(\xx;\yy)$ is a Boolean combination of $\Sigma_n$-formulas of $\Lang_c$, and let $k\ge 1$. For each $j < k$, choose a tuple of new variables $\zz\,^j$ of the same type as $\yy$. Define the formula
\[
\Upsilon_{\varphi(\xx;\yy), k} \defas \bigwedge_{ i < j < k} (\zz\,^{i} \neq \zz\,^{j}) \And \bigwedge_{j < k}\varphi(\xx, \zz\,^{j}),
\]
whose free variables we will partition in several different ways below.
This formula specifies $k$-many distinct realizations of the tuple $\yy$ in $\varphi(\xx;\yy)$, given an instantiation of $\xx$. Note that $\Upsilon_{\varphi(\xx;\yy), k}$ is also a Boolean combination of $\Sigma_n$-formulas of $\Lang_c$. 

For $j < k$, let $\tau_j \defas \xx\,\zz\,^0\,\cdots\,\zz\,^{j-1}\,\zz\,^{j+1} \cdots\,\zz\,^{k-1}$.
Note that $(\varphi(\xx; \yy), \aa, k) \in \CL_n^c$ if and only if 
\[
\bigl(\Upsilon_{\varphi(\xx;\yy), k}(\tau_j; \zz\,^j),
\ \aa\,\bb\,^0\,\cdots\bb\,^{j-1}\,\bb\,^{j+1}\,\cdots\,\bb\,^{k-1}\bigr)
\in \DCL_n^c
\]
for some $j < k$ and $\bb\,^0, \ldots, \bb\,^{j-1}, \bb\,^{j+1}, \ldots, \bb\,^{k-1}\in \cM_c$. By enumerating over all such parameters, and enumerating over all choices of $\varphi$ and $k$, we see that the desired set is $\Sigma^0_1$ in $\DCL_n^c$.
\end{proof}

\begin{lemma}
\label{CL for empty arguments is c.e. in DCL}
Uniformly in the parameters $c\in\Struct$ and $n\in\Nats$, the set
\[
\bigl\{\bigl(\varphi(\xx;\yy), \aa, k\bigr) \in \CL_n^c \st k = 0 \bigr\}
\]
is $\Sigma^0_1$ in $\DCL_n^c$.
\end{lemma}
\begin{proof}
Suppose $\varphi(\xx;\yy)$ is a Boolean combination of $\Sigma_n$-formulas of $\Lang_c$.
Let $\zz$ be a tuple of variables having the same type as $\yy$ 
and disjoint from $\xx\,\yy$. 
Define the formula
\[
\Psi_{\varphi(\xx;\yy)}(\xx\,\zz,\yy) \defas \varphi(\xx,\yy) \Or (\yy = \zz).
\]
Note that $\Psi_{\varphi(\xx;\yy)}(\xx\,\zz,\yy)$ is also a Boolean combination of $\Sigma_n$-formulas of $\Lang_c$.

Now suppose $\bbzero$ and $\bbone$ are distinct tuples of elements of $\cM_c$ having the same type as $\zz$. Then the following are equivalent:
\begin{itemize}
\item $\bigl(\Psi_{\varphi(\xx;\yy)}(\xx\,\zz; \yy), \aa\,\bbzero\bigr) \in \DCL_n^c$
and $\bigl(\Psi_{\varphi(\xx;\yy)}(\xx\,\zz; \yy), \aa\,\bbone\bigr) \in \DCL_n^c$; 
\vspace*{5pt}
\item $\bigl\{\bb \st \cM_c \models \varphi(\aa;\bb)\bigr\} = \emptyset$, i.e., $(\varphi(\xx;\yy), \aa, 0) \in \CL_n^c$.
\end{itemize}
The result is then immediate. 
\end{proof}

\begin{lemma}
\label{varphi-closure equivalent to varphi-acl + varphi-dcl}
Uniformly in the parameters $c\in\Struct$ and $n\in\Nats$, 
there are computable reductions in both directions between 
    $\ACL_n^c \coprod \DCL_n^c$ and $\CL_n^c$.  
\end{lemma}
\begin{proof}
It is immediate from the definitions that $\DCL_n^c$ is computable from $\CL_n^c$. 
Further, $\ACL_n^c$ is computable from $\CL_n^c$ as
\[
\ACL^c_n =\bigl\{(\varphi(\xx; \yy), \aa) \st \mathrm{there~exists~} k
\mathrm{~with}\ (\varphi(\xx; \yy), \aa, k) \in \CL^c_n \text{~and~} k \neq \infty\bigr\}
\]
and as
$(\varphi(\xx; \yy), \aa, k) \in \CL^c_n$ holds for a unique $k\in\Nats \cup \{\infty\}$.

Lemmas~\ref{CL for algebraic arguments is c.e. in DCL} and \ref{CL for empty arguments is c.e. in DCL} together tell us that the set
\[
\bigl\{(\varphi(\xx;\yy), \aa, k) \in \CL_n^c \st k \in \Nats\bigr\}
\]
is computably enumerable from $\DCL_n^c$. But $(\varphi(\xx;\yy), \aa, \infty) \in \CL_n^c$ if and only if  $(\varphi(\xx;\yy), \aa) \not \in \ACL_n^c$. Therefore when $\varphi(\xx;\yy)$ is a Boolean combination of $\Sigma_n$-formulas, and given $\aa \in \M_c$, we can compute from $\ACL_n^c$ whether or not $(\varphi(\xx;\yy), \aa, \infty) \in \CL_n^c$.
Further, if $(\varphi(\xx;\yy), \aa, \infty) \not\in \CL_n^c$, then we can compute from $\DCL_n^c$ the (unique) value of $k$ such that $(\varphi(\xx;\yy), \aa, k) \in \CL_n^c$. Hence $\CL_n^c$ is computable from $\ACL_n^c \coprod \DCL_n^c$.
\end{proof}

Note that by  Lemma~\ref{varphi-closure equivalent to varphi-acl + varphi-dcl} we are justified, from a computability-theoretic perspective, in restricting our attention to $\ACL$ and $\DCL$ 
(and their variants),
as opposed to $\CL$.


We next define a
closure operator 
with respect to a collection of formulas. We will use it to study computable aspects of 
algebraic and definable closure.
(See \cite[\S4.1]{MR1221741} for more details on the standard notions of algebraic and definable closure.)


\begin{definition}
\label{acl-dcl-transitive-closure}
Let $\Lang$ be a language,
let $\Phi$ be a set of first-order $\Lang$-formulas, and
let $\cN$ be an $\Lang$-structure.
Suppose
$B \subseteq\cN$ and
$S \subseteq \Nats \cup \{\infty\}$.
Define $\cl^n_{\Phi, \cN}(B, S)$ 
for $n \in \Nats$
by induction as follows. 
\vspace*{3pt}
\begin{itemize}
\item $\cl^0_{\Phi, \cN}(B, S) \defas B$,
\vspace*{8pt}

\item $\cl^1_{\Phi, \cN}(B, S) \defas B \cup \bigl\{b \in \cN \st
\mathrm{there~exists~}
\varphi \in \Phi
\mathrm{~and~a~tuple~}
\aa 
\mathrm{~from~}
B
\mathrm{~with~}
\\ \hspace*{140pt}
\bigl|\{\dd \st \cN\models\varphi(\aa;\dd)\}\bigr| \in S
\mathrm{~such~that~for~some~}
\bb\in \cN
\\ \hspace*{135pt}
\mathrm{~with~}
b \in \bb,
\mathrm{we~have~}
\cN \models \varphi(\aa; \bb)\bigr\}$,
\vspace*{8pt}

\item $\cl^{n+1}_{\Phi, \cN}(B, S) \defas \cl^1_{\Phi, c}\bigl(\cl^n_{\Phi, \cN}(B, S)\bigr)$.
\vspace*{3pt}
\end{itemize}
Let $\cl_{\Phi, \cN}(B, S) \defas \bigcup_{i \in \Nats}\cl^i_{\Phi, \cN}(B,S)$.
\end{definition}

When considering $\cl_{\Phi, \cN}(\pars, S)$,
it suffices to restrict our attention to 
the case where the argument is a \emph{finite} subset of $\cN$,
since for any $B\subseteq \cN$ we have 
\[
\cl_{\Phi, \cN}(B, S) = \bigcup
\bigl\{ 
\cl_{\Phi, \cN}(B_0,S) \st B_0 \mathrm{~is~a~finite~subset~of~} B\bigr\}.
\]

There are two instances of $S$ for which 
the operator $\cl_{\Phi, \cN}(\pars, S)$
is especially important model-theoretically.
The algebraic and definable closure operators in $\cN$ (with respect to $\Phi$) are given, respectively, by
\[
\acl_{\Phi,\cN}(\pars) \defas \cl_{\Phi, \cN}(\pars, \Nats)
\]
and
\[
\dcl_{\Phi,\cN}(\pars) \defas \cl_{\Phi, \cN}(\pars, \{1\}).
\]

The standard model-theoretic notions of first-order algebraic and definable closure in $\cN$ are $\acl_{\Lww(\Lang), \cN}(\pars)$ and $\dcl_{\Lww(\Lang), \cN}(\pars)$, respectively.
In these two cases, when $\Phi = \Lww(\Lang)$ and $S$ is either $\Nats$ or $\{1\}$, we have $\cl_{\Phi, \cN}(\pars,S) =  \cl^1_{\Phi, \cN}(\pars,S)$, i.e., the first step of the iterative process in Definition~\ref{acl-dcl-transitive-closure} is already idempotent. But this is not the case for every set $\Phi$ of formulas, and so to obtain a 
closure operator, we need the full iterative process.

Note that a key computability-theoretic distinction is whether or not $S$ is finite, and indeed one can easily check that all the upper and lower bounds proved in this paper for $\dcl_{\Phi, \cN}(\pars)$ also hold for 
$\cl_{\Phi, \cN}(\pars, S)$ for any finite $S$.

In order to study the computability-theoretic content
of the algebraic and definable closure operators,
we will consider
the following encodings of
their respective graphs.

\begin{definition}
Let $\Lang$ be a language,
let $\Phi$ be a set of first-order $\Lang$-formulas,
and let
$\cN$ be an $\Lang$-structure.
Define
\begin{eqnarray*}
\acl_{\Phi,\cN} \!\!&\defas&\!\! \bigl\{(a, A) \st a \in \acl_{\Phi, \cN}(A)
\mathrm{~and~}
A\mathrm{~is~a~finite~subset~of~}\cN\bigr\},\\
\dcl_{\Phi,\cN} \!\!&\defas&\!\! \bigl\{(a, A) \st a \in \dcl_{\Phi, \cN}(A)
\mathrm{~and~}
A\mathrm{~is~a~finite~subset~of~}\cN\bigr\}.
\end{eqnarray*}
\end{definition}

For $c\in\Struct$, write $\cl_{\Phi, c}(B,S)$ to denote $\cl_{\Phi, \cM_c}(B,S)$, and similarly with $\acl_{\Phi,c}(B)$, $\acl_{\Phi,c}$, $\dcl_{\Phi,c}(B)$, and $\dcl_{\Phi,c}$.

As can be seen from Definition~\ref{acl-dcl-transitive-closure}, the set
$\cl_{\Phi, c}(B, S)$ is closely related to $\CL^c$ via
the relation $Z_S$ on $\cM_c$ defined by
\[
Z_S
\defas \bigcup_{\psi \in \Phi}\left\{(\a, \b) \st \cM_c \models \psi(\aa;\bb) \ \text{~and~}(\psi(\xx;\yy),\aa, k) \in \CL^c\text{~with~}k \in S\right\}.
\]
For example, suppose 
every formula in $\Phi$ has just two free variables, and
let $U_S$ be the transitive closure in $\cM_c$ of $Z_S$.
Then 
\[
\acl_{\Phi, c}(B, S) =  \bigl\{b\in\cM_c \st \mathrm{there~exists~} a \in B
\mathrm{~for~which~} U_{\Nats}(a, b) \mathrm{~holds}\bigr\}
\]
and
\[
\dcl_{\Phi, c}(B, S) =  \bigl\{b\in\cM_c \st \mathrm{there~exists~} a \in B
\mathrm{~for~which~} U_{\{1\}}(a, b) \mathrm{~holds}\bigr\}.
\]



\section{Upper Bounds for Quantifier-Free Formulas}
\label{Section:Upper Bounds}

We now provide straightforward upper bounds on the complexity of $\ACL_0^c$, $\DCL_0^c$, $\acl_{\Phi, c}$, and $\dcl_{\Phi, c}$ for $c \in \Struct$ 
and $\Phi$ a computable set of quantifier-free first-order $\Lang_c$-formulas.

\begin{proposition}
\label{ACL is Sigma^0_2}
Uniformly in the parameter $c\in \Struct$,
the set $\ACL_0^c$ is a $\Sigma^0_2$ set.
\end{proposition}
\begin{proof}
Uniformly in $c \in \Struct$, a quantifier-free $\Lang_c$-formula $\varphi(\xx; \yy)$, and tuple $\aa\in \cM_c$ of the same type as $\xx$,
we can computably find an $e\in\Nats$
such that
$W_{e}$ equals 
\linebreak
$\bigl\{\bb\in \cM_c \st \cM_c \models \varphi(\aa;\bb)\bigr\}$
(where the tuples $\bb$ of this set are encoded in $\Nats$ in a standard way).

Further, $(\varphi(\xx; \yy), \aa) \in \ACL_0^c$ if and only if 
$\bigl\{\bb\in \cM_c \st \cM_c \models \varphi(\aa;\bb)\bigr\}$
is finite. Therefore $\ACL_0^c$ is $\Sigma^0_2$, as $\FINITE$ is $\Sigma^0_2$.
\end{proof}

\begin{proposition}
\label{DCL is intersection of Sigma^0_1 and Pi^0_1}
Uniformly in the parameter $c \in \Struct$, the set $\DCL_0^c$ is the intersection of a $\Pi^0_1$ set and a $\Sigma^0_1$ set (in particular, it is a $\Delta^0_2$ set).
\end{proposition}
\begin{proof}
Uniformly in $c\in\Struct$, the set of all pairs $(\varphi(\xx;\yy), \aa)$ such that 
\[
\cM_c \models (\forall \yyzero, \yyone)\ \bigl((\varphi(\aa, \yyzero) \And \varphi(\aa, \yyone)) \rightarrow (\yyzero = \yyone)\bigr)
\]
holds is a $\Pi^0_1$ set. Likewise, uniformly in $c\in\Struct$, the set of all pairs $(\varphi(\xx;\yy), \aa)$ such that there exists $\bb$ with $\cM_c \models \varphi(\aa; \bb)$ is a $\Sigma^0_1$ set. 
\end{proof}

As a consequence, $\DCL_0^c$ is computable from $\zerosinglejump$. 


\begin{proposition}
\label{prop:aclset-upperbound}
Uniformly in the parameter $c\in \Struct$
and an encoding of a computable set $\Phi$ of quantifier-free first-order $\Lang_c$-formulas,
the set $\acl_{\Phi, c}$
is $\Sigma^0_1$ in $\ACL^c_0$.
In particular, $\acl_{\Phi, c}$ is a $\Sigma^0_2$ set.
\end{proposition}
\begin{proof}
Let $A\subseteq\cM_c$ be a finite set.
Note that $b \in \acl_{\Phi, c}(A)$
if and only if there is a finite sequence
$b_0, \dots, b_{n-1} \in \cM_c$ where $b = b_{n-1}$
such that for each $i < n$, 
there exists a formula
$\varphi_i(\xx; \yy) \in \Phi$, a tuple
$\aa_i$ with entries from $A \cup \{b_j\}_{j < i}$, and 
a tuple $\dd_i\in\cM_c$ satisfying
\begin{itemize}
\item $(\varphi_i(\xx;\yy), \aa_i)\in \ACL^c_0$,
\item $\cM_c \models \varphi_i(\aa_i; \dd_i)$, and 
\item $b_i \in \dd_i$.
\end{itemize}
Hence, uniformly in $c$, the set $\acl_{\Phi, c}$
is $\Sigma^0_1$
in $\ACL^c_0$.
By Proposition~\ref{ACL is Sigma^0_2}, the set $\acl_{\Phi, c}$ is 
$\Sigma^0_2$.
\end{proof}

\begin{proposition}
\label{prop:dclset-upperbound}
Uniformly in the parameter $c\in \Struct$
and an encoding of a computable set $\Phi$ of quantifier-free first-order $\Lang_c$-formulas,
the set $\dcl_{\Phi, c}$ is $\Sigma^0_1$ in $\DCL^c_0$.
In particular, $\dcl_{\Phi, c}$ is 
a $\Sigma^0_2$ set.
\end{proposition}
\begin{proof}
Let $A\subseteq\cM_c$ be a finite set.
Note that $b \in \dcl_{\Phi, c}(A)$
if and only if there is a finite sequence
$b_0, \dots, b_{n-1} \in \cM_c$ where $b = b_{n-1}$
such that for each $i < n$, 
there exists a formula
$\varphi_i(\xx; \yy) \in \Phi$, a tuple
$\aa_i$ with entries from $A \cup \{b_j\}_{j < i}$, and 
a tuple $\dd_i\in\cM_c$ satisfying
\begin{itemize}
\item $(\varphi_i(\xx;\yy), \aa_i)\in \DCL^c_0$,
\item $\cM_c \models \varphi_i(\aa_i; \dd_i)$, and 
\item $b_i \in \dd_i$.
\end{itemize}
Hence, uniformly in $c$, the set $\dcl_{\Phi, c}$
is $\Sigma^0_1$
in $\DCL^c_0$.
By Proposition~\ref{DCL is intersection of Sigma^0_1 and Pi^0_1}, 
the set $\dcl_{\Phi, c}$ is 
$\Sigma^0_2$.
\end{proof}

\section{Lower Bounds for Quantifier-Free Formulas}
\label{Section:Lower Bounds}
We now 
prove lower bounds on $\ACL^c_0$, $\DCL^c_0$, $\acl_{\Phi, c}$, and $\dcl_{\Phi,c}$ that
show that the upper bounds in Section~\ref{Section:Upper Bounds} are tight.

In Propositions~\ref{Lower bound on ACL} and \ref{Lower bound on DCL}, we establish tightness of
the upper bounds
in 
Propositions~\ref{ACL is Sigma^0_2} and  \ref{DCL is intersection of Sigma^0_1 and Pi^0_1}, respectively.
Moreover, 
we do so using structures that have
nice model-theoretic properties ($\aleph_0$-categoricity for $\ACL_0$ and finite Morley rank for $\DCL_0$).

In Propositions~\ref{prop:aclset-lowerbound} and
\ref{prop:dclset-lowerbound}
we show tightness of the upper bounds in
Propositions~\ref{prop:aclset-upperbound} 
and \ref{prop:dclset-upperbound}, respectively.

We proved 
Propositions~\ref{prop:aclset-upperbound} 
and \ref{prop:dclset-upperbound}
by showing that when $\Phi$ is a computable collection of quantifier-free $\Lang_c$-formulas, 
the sets $\acl_{\Phi, c}$ and $\dcl_{\Phi, c}$ 
can be computably enumerated
from $\ACL^c_0$ and $\DCL^c_0$, respectively.
In Proposition~\ref{prop:bipartite-complicated},
we show that in general the converse does not hold, i.e., there is no information about $\ACL^c_0$ and $\DCL^c_0$ which can be uniformly deduced from $\acl_{\Phi, c}$ and $\dcl_{\Phi, c}$.


The structure we build in the proof of Proposition~\ref{Lower bound on ACL}
has unary relations 
$\const{i}$ and $\dconst{i}$, for $i\in\Nats$,
which each
hold of a single element.
These relations are not necessary to show tightness, but we will need them when we reuse this structure in the proof of
Proposition~\ref{prop:aclset-lowerbound}.

\begin{proposition}
\label{Lower bound on ACL}
There is a parameter $c \in \Struct$ such that the following hold.
\begin{itemize}

\item[(a)] 
$\Lang_c$ consists of, for each $i\in\Nats$, a sort $X_i$
and unary relation symbols $\const{i}$ and $\dconst{i}$ of sort $X_i$. Each of the $\const{i}$ and $\dconst{i}$ is instantiated by a single element of $\cM_c$.

\item[(b)] For each ordinal $\alpha$, the theory $T_c$ has $(|\alpha+1|^\omega)$-many models of size $\aleph_\alpha$. In particular, $T_c$ is $\aleph_0$-categorical. 

\item[(c)] 
$\bigl\{e \st   (X_e)^\cN \mathrm{~is~finite}   \mathrm{~for~every~} \cN\models T_c  \bigr\}
\equiv_1 \FINITE$.

\item[(d)]  $\ACL^c_0 \equiv_1 \FINITE$. In particular, $\ACL^c_0$ is a $\Sigma^0_2$-complete set. 
\end{itemize}
\end{proposition}
\begin{proof}
Let $\bigl((e_i, n_i)\bigr)_{i \in \Nats}$ be a computable enumeration without repetition of 
\[
\bigl\{(e, n) \st e, n \in \Nats\text{ and }\{e\}(n)\!\downarrow\bigr\}.
\]
Note that for each $\ell \in \Nats \cup \{\infty\}$, 
there are infinitely many programs that halt on exactly $\ell$-many inputs,
and so
there are infinitely many $e\in \Nats$ that are equal to $e_i$ for exactly $\ell$-many $i$.

Let $c \in \Struct$ be a code 
such that $\Lang_c$ is as in (a), and $\cM_c$ satisfies the following.
\vspace*{3pt}
\begin{itemize}
\item The underlying set of $\cM_c$ is $\Nats \cup (\{0, 1\} \times \Nats)$,
\vspace*{5pt}

\item $(\const{i})^{\cM_c} = \{(0, i)\}$ and $(\dconst{i})^{\cM_c} = \{(1,i)\}$  for $i\in\Nats$, and
\vspace*{5pt}

\item  $i \in (X_{e_i})^{\cM_c}$ for $i \in \Nats$.
\vspace*{3pt}
\end{itemize}

Each model of $T_c$ is
determined up to isomorphism by the number of elements in the instantiation of each sort.
Consider a model of $T_c$ of size $\aleph_\alpha$.
For each 
$j\in\Nats$ 
it has
$\aleph_0$-many 
sorts whose instantiations are of size $j$.
It also has $\aleph_0$-many whose instantiations are infinite, each of which may have size $\aleph_\beta$ for arbitrary $\beta \le \alpha$.
Hence (b) holds.

Note that
\[
(X_e)^{\cM_c} =
\{i \st e_i = e\} \cup \{(0, e), (1, e)\}.
\]
So for any countable $\cN \models T_c$ we have
\[
\bigl|(X_e)^{\cN}\bigr|
=
|W_e| +2.
\]
Hence $\FINITE$ is $1$-equivalent to  the set
$\bigl\{e \st (X_e)^{\cN} \mathrm{~is~finite}\bigr\}$,
which is equal to the set 
$\bigl\{e \st   (X_e)^\cN \mathrm{~is~finite}   \mathrm{~for~every~} \cN\models T_c  \bigr\}$,
proving (c).

Because all relation symbols in $\Lang_c$ are unary, any definable set is the product of definable sets that are each contained in the instantiation of a single sort. Further, given a countable $\cN\models T_c$, a finite $A\subseteq \cN$,
and an $i\in\Nats$, the definable sets (with parameters from $A$) in $(X_i)^\cN$ are Boolean combinations of 
$\{(\const{i})^\cN, (\dconst{i})^\cN\} \cup \bigl\{\{a\} \st a \in A
\cap (X_i)^\cN
\bigr\}$.

Hence $\ACL^c_0$ is $1$-equivalent to $\bigl\{e \st (X_e)^{\cM_c} \text{~is~finite}\bigr\}$
as well, establishing (d).
\end{proof}


We now show that the upper bound in Proposition~\ref{DCL is intersection of Sigma^0_1 and Pi^0_1} is tight.

\begin{proposition}
\label{Lower bound on DCL}
There is a parameter $c \in \Struct$ such that the following hold.

\begin{itemize}
\item[(a)] The language $\Lang_c$ has one sort and one binary relation symbol $E$.

\item[(b)] The structure $\cM_c$ has 
underlying set $\Nats$ and 
is a countable saturated model of $T_c$.

\item[(c)] For each ordinal $\alpha$, the theory $T_c$ has $(|\alpha + \w|)$-many
models of size $\aleph_\alpha$, and has finite Morley rank.

\item[(d)] There is a computable array $\bigl(U_{k, \ell}\bigr)_{k, \ell \in \Nats}$ of subsets of $\Nats$ such that each countable model of $T_c$ is isomorphic to the restriction of $\cM_c$ to the underlying set $U_{k, \ell}$ for exactly one pair $(k, \ell)$.

\item[(e)] If $\cN \cong \cM_c$ then uniformly in $\cN$ we can compute $\zerosinglejump$ from the set
\[
\bigl\{a \st \bigl|\{b \st \cN \models E(a;b)\}\bigr| = 1\bigr\}.
\]

\item[(f)] The set 
\[
\bigl\{a \st (E(x;y), a) \in \DCL_0^c\bigr\}
\]
has Turing degree $\zerosinglejump$.
\end{itemize}

\end{proposition}
\begin{proof}
Let $g \colon \Nats \to \{0, 1\}$ be the characteristic function of $\zerosinglejump$, i.e., such that $g(n) = 1$ if and only if $n \in \zerosinglejump$. 
As $\zerosinglejump$ is a $\Delta^0_2$ set, there is some computable
function $f\colon\Nats \times \Nats \to \{0, 1\}$ such that 
$\lim_{s \to \infty} f(n, s) = g(n)$
for all $n\in \Nats$.

We will construct $\cM_c$ in the language specified in (a) so as to satisfy the following axioms.

\begin{itemize}
\item $(\forall x)\ \neg E(x, x)$

\item $(\forall x, y)\ (E(x, y) \rightarrow E(y, x))$

\item $(\forall x)(\exists y)\ E(x, y)$

\item $(\forall x)(\exists^{\leq 2} y)\ E(x, y)$
\end{itemize}

These
axioms specify that 
$(\Nats, E^{\cM_c})$ is a graph
that is the union of chains.
In fact, we will construct $\cM_c$ so as to have infinitely many chains of certain finite orders, infinitely many $\Nats$-chains, and infinitely many $\Integers$-chains.

For $n\in \Nats$, let $\Prime{n}$ denote the $n$th prime number. We now construct $\cM_c$ with underlying set $\Nats$, in stages. \nl\nl
\ul{Stage $0$:}\nl
Let $\{N_i\}_{i \in \Nats} \cup \{Z_i\}_{i \in \Nats} \cup \{F\}$
be a uniformly computable partition of $\Nats$ into infinite sets.

For each $i \in \Nats$, let 
the induced subgraph on 
$N_i$ be an $\Nats$-chain, and let
the induced subgraph on $Z_i$ be a $\Integers$-chain. 
The only other edges 
will be between elements of $F$
(to be determined at later stages).\nl\nl
\ul{Stage $2s+1$:}\nl 
Let $a_s$ be the least element of $F$ that is not yet part of an edge. Create a finite chain of order $(\Prime{s})^{2 + f(s, s)}$ consisting of $a_s$ and other elements of $F$ not yet in any edge.\nl\nl%
\ul{Stage $2s+2$:}\nl 
For each $n \leq s$, we have two cases, based on the values of $f$:

\begin{itemize}
\item If $f(n, s) = f(n, s+1)$, do nothing.

\item Otherwise, if $f(n, s) \neq f(n, s+1)$, consider the (unique) chain whose order so far is $(\Prime{n})^k$ for some positive $k$. Extend this chain by $\bigl((\Prime{n})^{k+1} - (\Prime{n})^k\bigr)$-many elements of $F$ which are not yet in any edge, to obtain a chain that has order $(\Prime{n})^{2\ell + f(n, s+1)}$ for some $\ell \in \Nats$. 
\end{itemize}

\ \nl%
\indent The resulting graph is computable, as every vertex participates in at least one edge, and whether or not there is an edge between a given pair of vertices is determined by the first stage at which each vertex of the pair becomes part of some edge.

Observe that every element of $F$ is part of a chain of elements of $F$ whose order is some positive power of a prime, which moreover is the only chain in $\cM_c$ whose order is a power of that prime.

Every model of $T_c$ is determined up to isomorphism by the number of $\Nats$-chains and the number of $\Integers$-chains in it.  In a model of size $\aleph_\alpha$, there must be either $\aleph_\alpha$-many $\Nats$-chains and
$0$, $1$, $\ldots$, $\aleph_0$, $\ldots$, or $\aleph_\alpha$-many $\Integers$-chains, or vice-versa.
The countable saturated models of $T_c$ are those with $\aleph_0$-many $\Nats$-chains and $\aleph_0$-many $\Integers$-chains,
and since $\cM_c$ has $\aleph_0$-many $\Nats$-chains and $\aleph_0$-many $\Integers$-chains,
condition (b) holds. 
None of these $\Nats$-chains or $\Integers$-chains are first-order definable, and so
condition (c) holds. 

For condition (d), let $U_{k, \ell}\defas 
\bigcup_{i < k} N_{i} \cup \bigcup_{i < \ell} Z_{i} \cup F$.

Towards condition (e), note that for each $n\in\Nats$, there is a unique chain of order a power of $\Prime{n}$. Writing $(\Prime{n})^{j_n}$ for this order, we have $j_n\equiv g(n) \pmod{2}$.
An element $a \in \cN$ is one of the two ends of a finite chain or the beginning of an $\Nats$-chain if and only if $|\{b \st \cN\models E(a;b)\}| = 1$. So, from the set
$\{a \st |\{b \st \cN\models E(a;b)\}| = 1\}$ we can enumerate the orders of all finite chains, and hence can compute $g(n)$ for all $n$.

Finally, recall that $\DCL^c_0$ is computable from $\zerosinglejump$ and so the set
$\{a \st (E(x;y), a) \in \DCL_0^c\}$ is also computable from $\zerosinglejump$. Hence (f) follows from (e). 
\end{proof}


We now use the structure constructed in Proposition~\ref{Lower bound on ACL}
to prove 
that the bound in Proposition~\ref{prop:aclset-upperbound} is tight. 


\begin{proposition}
\label{prop:aclset-lowerbound}
Let $c\in\Struct$ be the parameter constructed in the proof of Proposition~\ref{Lower bound on ACL}.  Then there is
a computable set $\Xi$ of 
quantifier-free first-order
$\Lang_c$-formulas
such that 
if $\cN \cong \cM_c$, then uniformly in $\cN$ 
we can compute $\FINITE$ from $\acl_{\Xi, \cN}$ via a $1$-reduction relative to $\cN$.
In particular, for computable such $\cN$, the set $\acl_{\Xi, \cN}$ is $\Sigma^0_2$-complete.
\end{proposition}

\begin{proof}
For each sort $X_i$ in $\Lang_c$, let $\xi_i(x,y)$ be the $\Lang_c$-formula that asserts that $x$ and $y$ are each of sort $X_i$.
Let $\Xi \defas \{\xi_i(x;y)\}_{i \in \Nats}$.
Suppose $\cN \cong \cM_c$.
Then there exists an isomorphism $\tau\colon \cN \to \cM_c$
that is computable in $\cN$.
Let
$A\subseteq \cN$ be finite and $b\in \cN$.
Note that
$b \in
\acl_{\Xi, \cN}(A)$ if and only if there is some 
$a \in A$ and
$i\in\Nats$ 
for which
\[
\bigl(\xi_i(x;y), \tau(a)\bigr) \in  \ACL^c_0
\qquad \mathrm{and}\qquad
\cN \models \xi_i(a;b).
\]
By the choice of the code $c$, for each $i\in \Nats$, the unique element of
$(V_i)^\cN$
is in $\acl_{\Xi, \cN}\bigl((\const{i})^\cN\bigr)$
if and only if $X_i$ is finite, establishing the proposition.
\end{proof}


We now build a structure that shows that the bound in Proposition~\ref{prop:dclset-upperbound} is also tight.

\begin{proposition}
\label{prop:dclset-lowerbound}
There is a parameter $c \in \Struct$ such that 
$\Lang_c$ contains a 
ternary relation symbol $F$ and, letting
$\Gamma \defas \{F(x, y; z)\}$,
if $\cN \cong \cM_c$ then uniformly in $\cN$ we can compute $\FINITE$ from 
$\dcl_{\Gamma, \cN}$
via a $1$-reduction relative to $\cN$.
In particular, for computable such $\cN$, the set $\dcl_{\Gamma, \cN}$ is $\Sigma^0_2$-complete.
\end{proposition}
\begin{proof}
Let $\Lang$ be the (one-sorted) language consisting of
unary relation symbols
$A$, $B$, $C$, $D$, $H$,
a binary relation symbol $E$, and a ternary relation symbol $F$. 

We first define a computable $\Lang$-structure $\cY$ 
such that $\dcl_{\Gamma, \cY}$ is computable.  Write $Y$ for its underlying set, and write $\star$ for a distinguished element of $Y$.
The relation $D^{\cY}$ is the singleton $\{\star\}$,
and the other unary relations $A^{\cY}$, $B^{\cY}$, 
$C^{\cY}$, and
$H^{\cY}$
partition $Y \setminus \{\star\}$ into disjoint infinite sets.
Let $\{a_i\}_{i \in \Nats}$ 
and $\{b_i\}_{i \in \Nats}$ 
be enumerations of the elements of 
$A^{\cY}$
and
$B^{\cY}$, respectively.

The pair $(Y, E^{\cY})$ is a graph whose non-trivial connected components are finite chains $L_i$, for $i\in\Nats$, with the following properties.
For each $i\in\Nats$, 
the chain $L_i$ has order $i+2$. The degree-$1$ vertices of $L_i$ are $a_i$ and $b_i$, and its remaining vertices satisfy $H$ (chosen computably).
Every element of $H^{\cY}$ is in one such $L_i$, and no elements of $C^{\cY}$ or $D^{\cY}$ are.
We will define $(Y, F^{\cY})$ later.

Observe that 
the graph
$(A^{\cY} \cup B^{\cY} \cup H^{\cY}, E^{\cY})$ is rigid.
Furthermore, 
for any graph $\cP$ that is isomorphic to 
$(A^{\cY} \cup B^{\cY} \cup H^{\cY}, E^{\cY})$, 
the unique such
isomorphism is computable uniformly in $\cP$.

We will eventually use $\cY$ to build a computable $\Lang$-structure $\cM_c$ having the same underlying set $Y$, satisfying the statement of the Proposition.
The instantiations of 
$A$, $B$, $C$, $D$, $H$, and $E$ on 
$\cM_c$ and $\cY$ will agree.
They will also agree on $F$ restricted to 
$Y \setminus \{\star\}$.
We will 
encode
$\FINITE$ in $\cM_c$ via the behavior of $F$ on triples that include $\star$.

We now define $F^{\cY}$.
The first coordinate of any $F$-triple in $\cY$ will satisfy either $A$ or $C$.
It will be useful to think of $F^{\cY}$ as a collection, 
indexed by the first axis of $F^{\cY}$, of binary relations on $Y$:
for
$r \in Y$,
write $\cF_r$
to denote the relation
\[
\bigl\{ (s, t) \in Y \times Y \st \cY \models F(r, s, t) \bigr\}.
\]
For $r\in Y$, the pair
$(Y, \cF_r)$ will be a digraph whose edge set is either empty or forms a single path with initial vertex satisfying $A$. For such a path, if $\cF_r$ is finite, then the final vertex of the path will satisfy $B$; all vertices of the path that are neither initial nor final will satisfy $C$.

Partition $C^{\cY}$ into sets $\{P_{i, \infty}\}_{i \in \Nats} \cup \{P_{i, k}\}_{i, k \in \Nats}$ where for each $i \in \Nats$, the set $P_{i, \infty}$ is infinite and the set $P_{i, k}$ has size $k$.
For $i \in \Nats$, enumerate $P_{i, \infty}$ by $\{r_{i, \infty, j} \st j \in \Nats\}$. 

For $r \in Y \setminus (A^{\cY} \cup \bigcup_{i\in\Nats}P_{i, \infty})$, let $\cF_{r}$ be empty. 

For $i\in \Nats$, let
$(Y, \cF_{a_i})$ have a single non-trivial connected component, namely a single $\Nats$-path
with its initial vertex equal to $a_i$ and vertex set $\{a_i\} \cup P_{i, \infty}$, with 
$(a_i, r_{i,\infty,0})\in  \cF_{a_i}$  and 
$(r_{i,\infty,j}, r_{i, \infty, j+1})\in  \cF_{a_i}$ for $j\in\Nats$.

For each $i, k \in \Nats$, let $(Y, \cF_{r_{i, \infty, k}})$ have a single non-trivial connected component, namely a path of order $k+2$ with initial vertex $a_i$, final vertex $b_i$, and vertex set $\{a_i, b_i\} \cup P_{i, k}$.
Note that for all $i, k \in \Nats$, we have $r_{i, \infty, k} \in \dcl_{\Gamma, \cY}(\{a_i\})$, and further,
\[
\bigl|\bigl\{t \st \cY \models F(r_{i, \infty,k}, a_i;t)\bigr\}\bigr| = 1.
\]
We say that $P_{i, k}$ \emph{witnesses} that $b_i \in \dcl_{\Gamma, \cY}(\{a_i\})$. 
This completes the definition of $\cY$.

We are now ready to define $\cM_c$, a computable structure that has the same underlying set as $\cY$ and that agrees with $\cY$ on 
$Y \setminus \{\star\}$.  

As in the proof of  Proposition~\ref{Lower bound on ACL},
let $\bigl((e_i, n_i)\bigr)_{i \in \Nats}$ be a computable enumeration without repetition of 
\[
\bigl\{(e, n) \st e, n \in \Nats\text{ and }\{e\}(n)\!\downarrow\bigr\}.
\]
Let $\cM_c \models F(r, s, t)$ with $\star \in \{r, s, t\}$ hold if and only if
for some $i \in \Nats$ and $k \leq n_i$,
\[
(r,s,t) = (r_{e_i, \infty, k}, a_{e_i}, \star).
\]
Consequently, for $i\in\Nats$ and $k \leq n_i$ we have
\[
\bigl|\bigl\{t \st \cM_c \models F(r_{e_i, \infty,k}, a_{e_i};t)\bigr\}\bigr| = 2,
\]
and so $P_{e_i, k}$ \emph{does not} witness that $b_{e_i}$ is in $\dcl_{\Gamma, c}(\{a_{e_i}\})$.  

On the other hand, for $i\in\Nats$, if  for all $h \in \Nats$ with $e_{h} = e_i$ we have $k>n_{h}$ then the path
$P_{e_i, k}$ still witnesses that $b_{e_i}$ is in $\dcl_{\Gamma, c}(\{a_{e_i}\})$.

Let $\ell\in\Nats$. By construction, we have $b_\ell \in \dcl_{\Gamma, c}(\{a_\ell\})$
if and only if this fact is witnessed by $P_{\ell, j}$ for some $j \in \Nats$. By the above, there is some $j$ such that $P_{\ell, j}$ witnesses $b_\ell \in \dcl_{\Gamma, c}(\{a_\ell\})$  if and only if $\sup \{n \st
\{\ell\}(n)\!\downarrow\}$ is finite. 

Hence $\FINITE$, a $\Sigma^0_2$-complete set, is $1$-reducible to
\[
\bigl\{(a_\ell, b_\ell)\st b_\ell \in \dcl_{\Gamma, c}(\{a_\ell\})\bigr\}
\]
relative to $\cN$, as desired.
\end{proof}

In Propositions~\ref{prop:aclset-upperbound} and
\ref{prop:dclset-upperbound},
we provided upper bounds on
the difficulty of computing $\acl_{\Phi, c}$ from $\ACL^c_0$, and 
of computing $\dcl_{\Phi, c}$ from $\DCL^c_0$,
for $\Phi$ a computable set of quantifier-free first-order $\Lang_c$-formulas.
We now show that in general, merely knowing $\acl_{\Phi, c}$ and $\dcl_{\Phi, c}$ will not lower the difficulty of computing even the $\Phi$-fiber of $\ACL^c_0$ or $\DCL^c_0$. We do so by providing examples where the $\Phi$-fibers of $\ACL^c_0$ and of $\DCL^c_0$ are maximally complicated but $\acl_{\Phi, c}$ and $\dcl_{\Phi, c}$ are trivial.

\begin{proposition}
\label{prop:bipartite-complicated}
There are $c_0, c_1 \in \Struct$ such that the following hold.

\begin{itemize}
\item[(a)] The (one-sorted) language $\Lang_{c_0} = \Lang_{c_1}$ contains
a ternary relation symbol $F$
and a unary relation symbol $U$.

\item[(b)] $\cM_{c_0}$ and $\cM_{c_1}$ have the same underlying set $M$ and 
agree on all unary relations.

\item[(c)] Let $\psi(x,y,z)\defas F(x, y, z) \And \neg F(x, z, y)$, and write $\Psi = \{\psi(x,y;z)\}$. For any $A \subseteq M$,
\[
\acl_{\Psi, c_0}(A) = \dcl_{\Psi, c_1}(A) = 
\begin{cases}
M & \text{if~} A \cap U \neq \emptyset, \text{and}\\
 \emptyset & \text{if~}  A \cap U = \emptyset.
\end{cases}
\] 

\item[(d)] If $\cN \cong \cM_{c_0}$ 
then uniformly in $\cN$, the set $\FINITE$ is $1$-reducible to
the set
\[
\bigl\{(u, a)\st u \in U\text{~and~}\{b\st \cN \models \psi(u, a; b)\}\mathrm{~is~finite}\bigr\},
\]
relative to $\cN$.
In particular, $\FINITE \le_1 \ACL_0^{c_0}$, and so $\ACL_0^{c_0}$ is a $\Sigma^0_2$-complete set.

\item[(e)] If $\cN \cong \cM_{c_1}$
then uniformly in $\cN$ we can compute $\zerosinglejump$ from
the set
\[
\bigl\{(u, a)\st u \in U\text{~and~}\bigl|\{b\st \cN \models \psi(u, a; b)\}\bigr| = 1\bigr\}.
\]
In particular $\DCL_0^{c_1}$ is Turing equivalent to $\zerosinglejump$.
\end{itemize}
\end{proposition}
\begin{proof}
Let $\Lang'$ be the (one-sorted) language consisting of unary relation symbols $U$, $A$, $B$, $C$, $D$, $H$, a binary relation symbol $E$, and a ternary relation symbol $F$. Let $\cK$ be the sublanguage of $\Lang'$ consisting of the relation symbols $A$, $B$, $C$, $D$, $H$ and $E$. 

We begin by defining a computable $\Lang'$-structure $\cZ$.
The reduct of the structure $\cZ$ to $\cK$ will be the same as the reduct to $\cK$ of the structure $\cY$ in the proof of Proposition~\ref{prop:dclset-lowerbound}
(in particular, the underlying set of $\cZ$ is also $Y$).
This will imply that 
the graph
\linebreak
$(A^{\cZ} \cup B^{\cZ} \cup H^{\cZ}, E^{\cZ})$ is rigid, and that
for any graph $\cP$ that is isomorphic to 
$(A^{\cZ} \cup B^{\cZ} \cup H^{\cZ}, E^{\cZ})$, 
the unique such
isomorphism is computable uniformly in $\cP$.

The unary relation $U^{\cZ}$ consists of three elements $u_0, u_1, u_2$ where $u_0, u_1 \in C^{\cZ}$ and
$D^{\cZ} = \{u_2\}$. 
It remains to describe $F^{\cZ}$.

We will eventually build computable $\Lang'$-structures $\cM_{c_0}$ and $\cM_{c_1}$, each with underlying set $Y$, 
which
agree with $\cZ$ on the unary and binary relations, and are such that 
$F^{\cZ} \subseteq F^{\cM_{c_j}}$ for $j \in \{0,1\}$.
For $j \in \{0, 1\}$ we will construct $\cM_{c_j}$ such that if $(r, s, t) \in F^{\cM_{c_j}}$, then $r \in U^\cZ$.

We now describe $F^{\cZ}$. 
For any $(r,s,t)\in F^{\cZ}$ we will have $r\in U^\cZ$.
Define,
for $i \in \{0, 1, 2\}$, the binary relations
\[
\cF_i \defas \bigl\{(s, t) \in Y \times Y \st \cZ \models F(u_i, s, t)\bigr\}.
\]
For each $i$, the structure
$(Y, \cF_i)$ will be a digraph; further if $(s, t) \in \cF_{i}$ and $\{s, t\} \cap U^{\cZ} \neq \emptyset$ then $s = u_i$ and $t = u_k$, where $k \equiv i+1 \ (\mathrm{mod~} 3)$. 
In particular, for each $i$
there is a single $\cF_i$-edge in $U^{\cZ}$ and no other $\cF_i$-edge involves an element of $U^\cZ$.

To complete the description of $F^\cZ$, we now describe each $\cF_i$ outside $U^\cZ$.
Let 
the digraph $(Y \setminus U^{\cZ}, \cF_0)$ be any computable infinite path, and let $\cF_1$ be such that for $s, t \in Y \setminus U^{\cZ}$ we have $\cZ \models \cF_0(s, t)\leftrightarrow \cF_1(t, s)$. 
The digraph $(Y \setminus U^\cZ, \cF_2)$ has no edges,
i.e., $\cF_2 = \{(u_2, u_0)\}$. 
This completes the definition of $\cZ$.

Now we move towards defining $\cM_{c_i}$ for $i \in \{0,1\}$.
Suppose $\cG$ is a computable bipartite graph with underlying set $A^{\cZ} \cup  B^{\cZ}$ and underlying partition $\{A^{\cZ},  B^{\cZ}\}$, in the single-sorted language consisting of a single binary relation symbol $G$.
For such a $\cG$, define $\cZ(\cG)$ to be the $L'$-structure with underlying set $Y$ that agrees with $\cZ$ on all unary relations and $E$, and for which 
\[F^{\cZ(\cG)} = F^\cZ \cup \bigl\{(u_2, s,t) \st (s,t) \in G^{\cG}\bigr\}.\]
Let $f$ 
be any computable function which takes a code for computable bipartite graphs $\cG$ with underlying partition $\{A^{\cZ}, B^{\cZ}\}$ and returns a code for $\cZ(\cG)$. Similarly define $d$ to be any computable function such that $d(\cG)$ is a code for $\cG$. 

It is straightforward to check that for any
computable bipartite graph $\cG$ with underlying partition $\{A^\cZ, B^\cZ\}$, we have that $\cZ(\cG)$ is computable and 
\begin{eqnarray*}
\acl_{\Psi, \cZ} &=& \acl_{\Psi, \cZ(\cG)},\\
\dcl_{\Psi, \cZ} &=& \dcl_{\Psi, \cZ(\cG)}.
\end{eqnarray*}
In particular, if $\cG_0$ and $\cG_1$ are such graphs then (a), (b) and (c) hold for $c_0 = f(\cG_0)$ and $c_1 = f(\cG_1)$.

Observe that from
the set of 
pairs of the form $(\psi(x, y; z), (u_2,b))$ in $\ACL^{f(\cG)}_0$,
we can compute those of the form
$(G(x;y), a)$ in $\ACL^{d(\cG)}_0$.
Likewise, from the set of pairs of the form
$(\psi(x, y; z), (u_2,b))$ in $\DCL^{f(\cG)}_0$
we can compute those of the form
$(G(x;y), a)$ in $\DCL^{d(\cG)}_0$.

To finish the proof, we now choose  $\cG_0$ and $\cG_1$.
For $i \in \Nats$, let $a_i$ be the unique element of $A^{\cZ}$ in a chain of order $i+2$ in $(Y, E^{\cZ})$.

Let $\cG_0$ be any computable bipartite graph 
with underlying partition $\{A^\cZ, B^\cZ\}$ 
such that for each $e \in \Nats$, the vertex $a_{e}$ is adjacent to
$\bigl|\{n \st \{e\}(n)\!\downarrow\}\bigr|$-many elements
in $B^\cZ$.
If $\cN$ is any computable structure isomorphic to $f(\cG_0)$, then $\FINITE$ is $1$-reducible to
\[
\bigl\{(u, a)\st u \in U\text{~and~}\{b\st \cN \models 
F(u, a; b)\}\mathrm{~is~finite}\bigr\}
\]
relative to $\cN$.
Hence $\FINITE$ is $1$-reducible to $\ACL^{d(\cG_0)}_{0}$, and so
(d) holds. 

Let
$\cG_1$ be any computable bipartite graph 
with partition $\{A^\cZ, B^\cZ\}$ 
such that for each $e \in \Nats$, the vertex $a_{e}$ is adjacent a single element of $B^\cZ$ if
$\{e\}(0)\!\downarrow$, and to no elements otherwise.
Then if $\cN$ is any computable structure isomorphic to $f(\cG_1)$, we can compute $\FINITE$ from 
\[
\bigl\{(u, a)\st u \in U\text{~and~}|\{b\st \cN \models F(u, a; b)\}| = 1\bigr\}.
\]
Hence $\DCL^{d(\cG_1)}_{0}$ can compute $\zerosinglejump$, and so
(e) holds, completing the proof. 
\end{proof}

\section{Bounds for Boolean Combinations of $\Sigma_n$-Formulas}
\label{Section:Boolean Combinations}

In Sections~\ref{Section:Upper Bounds} and \ref{Section:Lower Bounds} 
we studied,
for $c \in \Struct$,
the complexity of
$\ACL_0^c$ and $\DCL_0^c$, and of 
$\acl_{\Phi, c}$ and $\dcl_{\Phi, c}$
where  $\Phi$ is a computable set of quantifier-free first-order $\Lang_c$-formulas.
We now study the complexity of $\ACL^c_n$ and $\DCL^c_n$, for arbitrary $n\in\Nats$, and of 
$\acl_{\Phi, c}$ and $\dcl_{\Phi, c}$
where $\Phi$ is a computable set of first-order $\Lang_c$-formulas of quantifier rank at most $n$.

Morleyization is a technique for translating
a structure in a given language to a new structure, in a new language, that has quantifier elimination but the same definable sets. This is done
by introducing new relation symbols to take the place of existing formulas.
The following lemma is a computable version of this standard method.
The proof is straightforward.
\begin{lemma}
\label{Computable content of Morleyizing}
Let $\Lang$ be a computable language and $\cA$ a computable $\Lang$-structure. For each $n \in \Nats$ there is a computable language $\KLang_n$
and a $\zerojump{n}$-computable $\KLang_n$-structure $\cA_n$ such that the following hold.
\begin{itemize}
\item $\Lang \subseteq \KLang_n \subseteq \KLang_{n+1}$. 

\item $\cA$ is the reduct of  $\cA_n$ to the language $\Lang$.

\item For each first-order $\KLang_n$-formula $\varphi$ there is a first-order $\Lang$-formula $\psi_\varphi$ (of the same type as $\varphi$) such that 
\[
\cA_n \models (\forall x_0, \dots, x_{k-1})\ \varphi(x_0, \dots, x_{k-1}) \leftrightarrow \psi_\varphi(x_0, \dots, x_{k-1}),
\]
where $k$ is the number of free variables of $\varphi$.

\item For each first-order $\Lang$-formula $\psi$, if $\psi$ is a Boolean combination of $\Sigma_n$-formulas then there is a first-order quantifier-free $\KLang_n$-formula $\varphi_\psi$ (of the same type as $\psi$) such that 
\[
\cA_n \models (\forall x_0, \dots, x_{k-1})\ \psi(x_0, \dots, x_{k-1}) \leftrightarrow \varphi_\psi(x_0, \dots, x_{k-1}),
\]
where $k$ is the number of free variables of $\psi$.
\end{itemize}
\end{lemma}

We now use Lemma~\ref{Computable content of Morleyizing} to extend
our earlier results about
algebraic and definable closure for quantifier-free formulas
to formulas of higher quantifier rank; this comes
at the expense of
greater computability-theoretic
complexity.

\begin{corollary}
Let $n \in \Nats$. Uniformly in $c \in \Struct$,  we have that
\begin{itemize}
\item[(a)] $\ACL^c_{n}$ is a $\Sigma^0_{n+2}$ set, and

\item[(b)] $\DCL^c_{n}$ is a $\Delta^0_{n+2}$ set. 
\end{itemize}
Further, uniformly in $c\in\Struct$  and in 
a computable collection $\Phi$ of first-order $\Lang_c$-formulas of quantifier rank at most $n$, we have that
\begin{itemize}
\item[(c)] $\acl_{\Phi, c}$ is a $\Sigma^0_{n+2}$ set, and

\item[(d)] $\dcl_{\Phi, c}$ is a $\Sigma^0_{n+2}$ set.
\end{itemize}
\end{corollary}

\begin{proof}
By Lemma~\ref{Computable content of Morleyizing}, we know that $\ACL_{n}$ 
and $\DCL_{n}$ 
are equivalent to the relativization, 
to the class of structures computable in $\zerojump{n}$,
of $\ACL_0$ and $\DCL_0$, respectively.
Therefore by Propositions~\ref{ACL is Sigma^0_2} and
\ref{DCL is intersection of Sigma^0_1 and Pi^0_1}, 
$\ACL^c_{n}$ is a $\Sigma^0_2(\zerojump{n})$ set and $\DCL^c_{n}$ is a $\Delta^0_2(\zerojump{n})$ set and so (a) and (b) hold.

Similarly, (c) and (d) hold by Propositions~\ref{prop:aclset-upperbound} and~\ref{prop:dclset-upperbound}.
\end{proof}

In Theorem~\ref{Lower bounds on non-qf formulas} we will show that these bounds are tight. 
Towards this fact, we will need a definition and the technical results in Lemma~\ref{Bounds on non-qf formulas:Claim 1} and Proposition~\ref{Bounds on non-qf formulas} below.

Let $(\Nats, \Succ)$ be the digraph
with underlying set $\Nats$
where
$\Succ$ is the graph of the successor function on $\Nats$,
i.e., such that $\Succ(k, \ell)$ holds precisely when $\ell = k + 1$.

\begin{definition}
Suppose that $\Lang$ is a language containing a sort $N$ and a 
relation symbol $S$ of type $N \times N$.
Let $\cA$ be an $\Lang$-structure.
We call $(N^\cA, S^\cA)$ a \defn{copy of $\Nats$}
when there is an isomorphism between
$(N^\cA, S^\cA)$ and
$(\Nats, \Succ)$. 

Note that any such isomorphism is necessarily unique.
Given $\ell \in \Nats$, we write $\hat{\ell}$ to denote the corresponding element of $N^\cA$ under this isomorphism.
\end{definition}

\begin{lemma}
\label{Bounds on non-qf formulas:Claim 1}
Let $\Lang$ be a language containing a sort $N$ 
and a relation symbol $S$ of type $N \times N$ (and possibly other sorts and relation symbols). 
Let $\cA$ be an $\Lang$-structure such that 
$(N^{\cA}, S^{\cA})$ is a copy of $\Nats$.
Let $k\in\Nats$ and let $\gamma(\x, y)$
be an $\Lang$-formula that is a Boolean combination of $\Sigma_k$-formulas, where $\x$ is of some type $X$, and $y$ has sort $N$.

Suppose that
\[
\cA \models
(\forall \xx : X) (\exists^{\leq 1}y : N)(\exists z : N)\ S(y, z) \And \bigl(\gamma(\xx, y) \leftrightarrow \neg \gamma(\xx, z)\bigr).
\]

Let $f \colon X^\cA \times \Nats \to \{\True, \False\}$ be the function where 
\[
\cA \models \gamma(\aa, \hat{\ell})
\qquad
\text{if and only if}
\qquad
f(\aa, \ell) = \True.
\]
Note that $\lim_{\ell \to \infty}f(\aa, \ell)$ exists for all $\aa \in X^\cA$.

There is a first-order $\Lang$-formula $\gamma'(\xx)$, where $\xx$ is of type $X$,
such that $\gamma'$ is a Boolean combination of $\Sigma_{k+1}$-formulas and 
for all $\aa \in X^{\cA}$,
\[
\cA \models \gamma'(\aa)
\qquad
\text{if and only if}
\qquad
\lim_{\ell \to \infty} f(\aa, \ell) = \True.
\]
\end{lemma}
\begin{proof}
Define the formula $\gamma'$ by
    \[
\gamma'(\xx) \defas 
\bigl[(\forall y : N)\ \gamma(\xx, y) \bigr]
\ \ \Or \ \ \bigl[(\exists y, z  : N)\ \bigl(
S(y,z)
\And
\neg \gamma(\xx, y) \And \gamma(\xx, z) 
\bigr) \bigr].
\]

Clearly $\gamma'$ is a Boolean combination of $\Sigma_{k+1}$-formulas and has the desired property. 
\end{proof}

\begin{proposition}
\label{Bounds on non-qf formulas}
Let $n\in \Nats$. Let $\Lang$ be a language containing a sort $N$ 
and a relation symbol $S$ of type $N \times N$ 
(and possibly other sorts and relation symbols). Suppose $\cA$ is an $\Lang$-structure that is computable in $\zerojump{n}$ and such that 
$(N^\cA, S^\cA)$ is a computable 
copy of $\Nats$.
Then there is a computable language $\Lang^+$ 
and a computable $\Lang^+$-structure $\cA^+$ with the same underlying set as $\cA$
such that for every quantifier-free first-order $\Lang$-formula  $\eta$ 
in which $S$ does not occur,
there is a 
first-order
$\Lang^+$-formula $\varphi_\eta$ that is a Boolean combination of $\Sigma_n$-formulas such that $\eta^\cA = (\varphi_\eta)^{\cA^+}$.
\end{proposition}

\begin{proof}
We begin by defining, for relation symbols in $\Lang$ other than $S$, certain auxiliary functions.

For $R$ a relation symbol in $\Lang$ other than $S$,
let $X$ be its type.
For every $k\in\Nats$ such that $0 \le k \le n$, 
inductively define the
$\zerojump{n-k}$-computable 
function
$f_{R,k}\colon X^\cA \times \Nats^{k} \to \TrueFalse$
satisfying the following,
for all $\aa \in X^\cA$.
\begin{itemize}
\item $f_{R,0}(\aa) = \True$
\ if and only if\ $\cA \models R(\aa)$.

\item Suppose $k\ge 1$ and let $(\ell_0, \dots, \ell_{k-2}) \in \Nats^{k-1}$.
There is at most one $\ell_{k-1} \in \Nats$ for which
\[
f_{R,k}(\aa, \ell_0, \dots, \ell_{k-2}, \ell_{k-1}) \neq f_{R,k}(\aa, \ell_0, \dots, \ell_{k-2}, \ell_{k-1} + 1).
\]

Further,
\[
f_{R,k-1}(\aa, \ell_0, \dots, \ell_{k-2}) = \lim_{\ell_{k-1} \to \infty}\ f_{R,k}(\aa, \ell_0, \dots, \ell_{k-2}, \ell_{k-1}).
\]
\end{itemize}

Next, we define the computable language $\Lang^+$ as follows.
\begin{itemize}
\item $\Lang^+$ has the same sorts as $\Lang$.

\item For each relation symbol $R \in \Lang$ other than $S$, there is a relation symbol $R^+ \in \Lang^+$ of type $X \times N^n$, where $X$ is the type of $R$.
The language $\Lang^+$ also contains a relation symbol $S$ of type $N \times N$. These are the only relation symbols in $\Lang^+$.
\end{itemize}

Now define the computable $\Lang^+$-structure $\cA^+$ as follows.
\begin{itemize}
\item $\cA^+$ has the same underlying set as $\cA$,
and sorts are instantiated on the same sets in $\cA^+$ as in $\cA$.

\item 
$S^{\cA^+} = S^{\cA}$.

\item For each $R \in \Lang$ other than $S$, each tuple $\aa \in X^{\cA^+}$ where $X$ is the type of $R$, and any $\ell_0, \dots, \ell_{n-1} \in \Nats$, we have
\end{itemize}
\[
\cA^+ \models R^+(\aa, \hat{\ell_0}, \dots, \hat{\ell_{n-1}}) 
\quad
\text{if and only if}
\quad
f_{R,n}(\aa, \ell_0, \dots, \ell_{n-1}) = \True.
\]
(Recall that for $\ell \in\Nats$, we have defined $\hat{\ell}\in N^{\cA^+}$ to be the $\ell^{\mathrm{th}}$ element of the 
copy of $\Nats$.)

We now build, for each relation symbol $R \in \Lang$ other than $S$, an $\Lang^+$-formula $\varphi_R$, as follows.
First apply Lemma~\ref{Bounds on non-qf formulas:Claim 1} (with $k=0$) to $\cA^+$ and the $\Lang^+$-formula 
\[
\gamma_0(\x y_0 \cdots y_{n-2}, y_{n-1}) \defas R^+(\x, y_0, \ldots, y_{n-1})
\]
(where $\x$ has type $X$ and each $y_i$ has type $N$)
to obtain an $\Lang^+$-formula 
    $\gamma_0'(\x y_0 \cdots y_{n-2})$ that is a Boolean combination of $\Sigma_1$-formulas. Next apply Lemma~\ref{Bounds on non-qf formulas:Claim 1} again (with $k=1$) to $\cA^+$ and the $\Lang^+$-formula 
\[
\gamma_1(\x y_0 \cdots y_{n-3}, y_{n-2}) \defas \gamma_0'(\x y_0 \cdots y_{n-2})
\]
to obtain an $\Lang^+$-formula $\gamma_1'(\x y_0\cdots y_{n-3})$ that is a Boolean combination of $\Sigma_{2}$-formulas. Proceed in this way for $k = 2, \ldots, n-1$,
to obtain an $\Lang^+$-formula
$\varphi_R(\x) \defas \gamma_{n-1}'(\x)$ 
that is a Boolean combination of $\Sigma_n$-formulas for which $R^\cA = (\varphi_R)^{\cA^+}$. 

We can now extend the definition of $\varphi_\psi$ to quantifier-free formulas $\psi$ by induction, where $\varphi_{\neg \psi}$ is $\neg \varphi_\psi$, where $\varphi_{\psi_0 \And \psi_1}$ is $\varphi_{\psi_0} \And \varphi_{\psi_1}$,
and where $\varphi_{\psi_0 \Or \psi_1}$ is $\varphi_{\psi_0} \Or \varphi_{\psi_1}$.
\end{proof}

Combining 
Proposition~\ref{Bounds on non-qf formulas}
with results from Section~\ref{Section:Lower Bounds}, we obtain the following. 

\begin{theorem}
\label{Lower bounds on non-qf formulas}
For each $n \in \Nats$, the following hold.
\begin{itemize}
\item[(a)] There exists $a \in \Struct$ such that $\ACL^a_n$ is a $\Sigma^0_{n+2}$-complete set.

\item[(b)] There exists $b \in \Struct$ such that
$\DCL^b_n \equiv_\mathrm{T} \zerojump{n+1}$. 

\item[(c)] There is a computable set $\Phi$ of 
first-order
$\Lang_a$-formulas, all of quantifier rank at most $n$ such that $\acl_{\Phi, a}$
is a $\Sigma^0_{n+2}$-complete set, where $a\in\Struct$ is as in (a).

\item[(d)] 
There exists $d \in \Struct$ and
a computable set $\Theta$ of first-order $\Lang_d$-formulas, all of quantifier rank at most $n$, 
such that $\dcl_{\Theta, d}$ is a $\Sigma^0_{n+2}$-complete set.
\end{itemize}
\end{theorem}

\begin{proof}
We first prove (a) and (c).
Let $\cP$ be the relativization to the oracle $\zerojump{n}$ of the
computable structure $\cM_c$ from
the statement of
Proposition~\ref{Lower bound on ACL}, 
and call its language $\Lang$.
Consider the set $\Xi$ of quantifier-free
formulas from
Proposition~\ref{prop:aclset-lowerbound}.
Let $\Lang^*$ be the language that extends $\Lang$ by a new sort $N$ and a relation symbol $S$ of type $N \times N$.
Consider the $\Lang^*$-structure $\cP^*$ 
whose restriction to $\Lang$ is
$\cP$
and such that
$(N^{\cP^*}, S^{\cP^*})$ is a computable 
copy of $\Nats$ (instantiated on the new set of elements $N^{\cP^*}$).
Let $a\in\Struct$ be such that $\cM_a$ is the computable structure $(\cP^*)^+$ obtained from 
Proposition~\ref{Bounds on non-qf formulas} (when $\cA = \cP^*$),
and let 
$\Phi \defas \bigl\{ \varphi_\eta \st  \eta \in \Xi \bigr\}$ be the corresponding set of first-order $(\Lang_a)^+$-formulas, each partitioned in the same way as in $\Xi$.
Then $\ACL^a_n$ and $\acl_{\Phi, a}$ are $\Sigma^0_{n+2}$-complete sets, establishing (a) and (c).

Towards (b), let $\cQ$ be the 
relativization to $\zerojump{n}$ of the
structure $\cM_c$ from Proposition~\ref{Lower bound on DCL}.
Consider the structure $\cQ^*$ obtained from $\cQ$ by augmenting it by 
$(N^{\cQ^*}, S^{\cQ^*})$,
a new computable 
copy 
of $\Nats$, 
as in the proof of (a) and (c) above.
Let $b \in \Struct$ be such that 
$\cM_b$ is the computable structure $(\cQ^*)^+$ obtained by applying
Proposition~\ref{Bounds on non-qf formulas} to $\cQ^*$.
Then $\DCL^b_n \equiv_\mathrm{T} \zerojump{n+1}$.

We now prove (d). 
Let $\cM_c$ 
and $F(x,y,z)$
be as in
Proposition~\ref{prop:dclset-lowerbound},
and let $\cR$ be the relativization of $\cM_c$ to $\zerojump{n}$.
Consider the structure $\cR^*$ obtained by augmenting $\cR$ 
by a computable copy of $\Nats$, as above.
Let $d\in\Struct$ be such that $\cM_d$ is
the computable structure $(\cR^*)^+$
obtained by applying Proposition~\ref{Bounds on non-qf formulas} to $\cR^*$,
and let $\Theta \defas \{ \varphi_F(x, y; z) \}$.
Then $\dcl_{\Theta, d}$ is a $\Sigma^0_{n+2}$-complete set.
\end{proof}

Note that the structures constructed in Theorem~\ref{Lower bounds on non-qf formulas} (a) and (b) do not obviously have the nice model-theoretic properties ($\aleph_0$-categoricity or finite Morley rank) that those constructed in Propositions~\ref{Lower bound on ACL} and \ref{Lower bound on DCL} do, because the application of Proposition~\ref{Bounds on non-qf formulas} 
encodes a copy of $\Nats$ in a way that
makes their theories more elaborate.
Nor is it obvious that the structures constructed in 
Theorem~\ref{Lower bounds on non-qf formulas} (c) and (d) have these nice model-theoretic properties,
because they derive from the structures constructed in Propositions~\ref{prop:aclset-lowerbound} and \ref{prop:dclset-lowerbound}, which themselves do not obviously have these properties.
This leads us to the following question.

\begin{question}
Is there some $c \in \Struct$ such that $\ACL^c_n$ is 
$\Sigma^0_{n+2}$-complete or 
$\DCL^c_n \equiv_\mathrm{T} \zerojump{n+1}$,
and $\cM_c$ is nice model-theoretically (e.g., $\aleph_0$-categorical, strongly minimal, stable, etc.)?

Similarly, is there some $c \in \Struct$ and computable set $\Psi$ of first-order $\Lang_c$-formulas, all of quantifier rank at most $n$, such that
either $\acl_{\Psi,c}$ or $\dcl_{\Psi, c}$ is a 
$\Sigma^0_{n+2}$-complete set
and $\cM_c$ is nice model-theoretically?
\end{question}

\section*{Acknowledgements}
The authors would like to thank 
Sergei Artemov, Valentina Harizanov, Anil Nerode, 
and the anonymous referees 
for helpful comments.
The third author's work on this paper was partially supported by the National Science Foundation under Grant No.\ DMS-1928930 while she was in residence at the Mathematical Sciences Research Institute in Berkeley, California, during the Fall 2020 semester. 


\providecommand{\bysame}{\leavevmode\hbox to3em{\hrulefill}\thinspace}
\providecommand{\MR}{\relax\ifhmode\unskip\space\fi MR }
\providecommand{\MRhref}[2]{%
  \href{http://www.ams.org/mathscinet-getitem?mr=#1}{#2}
}
\providecommand{\href}[2]{#2}

\end{document}